\documentclass[10pt]{amsart}
\usepackage{amssymb}

\newtheorem{proposition}{Proposition}[section]
\newtheorem{lemma}[proposition]{Lemma}
\newtheorem{corollary}[proposition]{Corollary}
\newtheorem{theorem}[proposition]{Theorem}

\theoremstyle{definition}

\newcommand{\selabel}[1]{\label{se:#1}}
\newcommand{\seref}[1]{Section~\ref{se:#1}}

\def\<{\leqslant}
\def\>{\geqslant}
\def\a{\alpha}

\def\O{\Omega}

\def\ol{\overline}

\def\t{\triangle}
\def\e{\varepsilon}
\def\oo{\infty}

\def\s{\sigma}

\def\ot{\otimes}
\def\ra{\rightarrow}

\date{}

\begin{document}
\title{Green rings of Drinfeld Doubles of Taft algebras}
\author{Hua Sun}
\address{School of Mathematical Science, Yangzhou University,
	Yangzhou 225002, China}
\email{997749901@qq.com}
\author{Hassan Suleman Esmael Mohammed}
\address{School of Mathematical Science, Yangzhou University,
Yangzhou 225002, China}
\email{esmailhassan313@yahoo.com}
\author{Weijun Lin}
\address{School of Mathematical Science, Yangzhou University,
	Yangzhou 225002, China}
\email{wjlin@yzu.edu.cn}
\author{Hui-Xiang Chen}
\address{School of Mathematical Science, Yangzhou University,
	Yangzhou 225002, China}
\email{hxchen@yzu.edu.cn}
\begin{abstract}
In this article, we investigate the representation ring (or Green ring) of the Drinfeld double $D(H_n(q))$
of the Taft algebra $H_n(q)$, where $n$ is an integer with $n>2$
and $q$ is a root of unity of order $n$. It is shown that
the Green ring $r(D(H_n(q)))$ is a commutative
ring generated by infinitely many elements subject to certain relations.
\end{abstract}
\maketitle

\section{\bf Introduction}\selabel{1}

The study of representation rings (or Green rings) of Hopf algebras
has been revived recently. In \cite{ChVOZh} and \cite{LiZhang}, the authors investigated respectively
the representation rings of Taft algebras and generalized Taft algebras based on
the decomposition rules of tensor product modules given by Cibils \cite{Cib}.
The representation ring $r(H_n(q))$ of a Taft algebra $H_n(q)$ is isomorphic to a polynomial ring in
two variables modulo two relations.
Zhang, Wu, Liu and Chen \cite{ZWLC} studied the ring structure of the Grothendieck group
of the quantum Drinfeld double $D(H_n(q))$ of $H_n(q)$.
When $n=2$, the Taft algebra $H_2(-1)$ is the same as $H_4$, the Sweedler's 4-dimensional Hopf algebra
(see \cite{Sw, Ta}). Chen \cite{Ch5} gave the decomposition rules of tensor product modules
over $D(H_4)$ and described the Green ring $r(D(H_4))$.
$r(D(H_4))$ was also studied by Li and Hu in \cite{LiHu}.
In \cite{ChHasSun}, we studied the decomposition rules for tensor product modules over
$D(H_n(q))$ for $n\>3$.
In \cite{ChHasLinSun}, we computed the projective class rings of $D(H_n(q))$ and its two
2-cocycle deformations, and compared their representation types and some other properties
for $n\>3$.

In this article, we continue the study of the quantum Drinfeld doubles $D(H_n(q))$
of Taft algebras $H_n(q)$, but concentrate on the structure of the Green rings $r(D(H_n(q)))$.
The article is organized as follow. In \seref{2},
we recall the structure of the Hopf algebra $H_n(1,q)$, which is isomorphic to $D(H_n(q))$.
We also recall the indecomposable modules over $H_n(1,q)$.
In \seref{3}, we investigate the Green ring $r(H_n(1,q))$ for $n\>3$.
A minimal set of generators of the ring $r(H_n(1,q))$ is given.
The relations satisfied by the generators are also given. It is shown that
$r(H_n(1,q))$ is isomorphic to a quotient ring of the polynomial ring in infinitely many
variables modulo some ideal.

\section{\bf Preliminaries}\selabel{2}

Throughout, let $k$ be an algebraically closed field. Unless
otherwise stated, all algebras and Hopf algebras are
defined over $k$; all modules are finite dimensional and left modules;
dim and $\otimes$ denote ${\rm dim}_k$ and $\otimes_k$,
respectively. The references \cite{Ka, Mon, Sw} are basic references for the theory of Hopf algebras and quantum groups.
The readers can refer \cite{ARS} for the representation theory of algebras.
Let $\mathbb Z$ denote the set of all integers, and ${\mathbb Z}_n={\mathbb Z}/n{\mathbb Z}$
for an integer $n$.

\subsection{Green ring}\selabel{2.1}
~~

For an algebra $A$, we denote the category of finite dimensional $A$-modules by mod$A$.
For any $M\in{\rm mod}A$ and any $n\in\mathbb Z$ with $n\>0$, $nM$ denotes the
direct sum of $n$ copies of $M$. Then $nM=0$ if $n=0$.
Let l$(M)$ denote the length of $M$.

For a Hopf algebra $H$, mod$H$ is a tensor (or monoidal) category \cite{Ka, Mon}.
If $H$ is a quasitriangular Hopf algebra, then $U\ot V\cong V\ot U$ for any modules
$U$ and $V$ over $H$. We have already known that the quantum Drinfeld double $D(H)$ of a finite dimensional Hopf algebra
$H$ is always quasitriangular and symmetric (see \cite{Lo, ObSch, Ra}).

The Green ring $r(H)$ of a Hopf algebra $H$ is defined to be the abelian group generated by the
isomorphism classes $[V]$ of $V$ in mod$H$
modulo the relations $[U\oplus V]=[U]+[V]$, $U, V\in{\rm mod}H$. The multiplication of $r(H)$
is determined by $[U][V]=[U\ot V]$, the tensor product of $H$-modules. Then $r(H)$ is an associative ring
with identity.
The projective class ring $r_p(H)$ of $H$ is defined to be the subring of $r(H)$ generated by
projective modules and simple modules (see \cite{Cib99}).
Notice that $r(H)$ is a free abelian group with a $\mathbb Z$-basis
$\{[V]|V\in{\rm ind}(H)\}$, where ${\rm ind}(H)$ denotes the category
of indecomposable objects in mod$H$.

\subsection{Drinfeld doubles of Taft algebras}\selabel{1.2}
~~

The quantum doubles (or Drinfeld doubles) $D(H_n(q))$ of Taft algebra $H_n(q)$ and their finite dimensional representations
were studied in \cite{Ch1, Ch2, Ch3, Ch4}. 

Let $n\>2$ and let $q\in k$ be a root of unity of order $n$. $H_n(q)$ is
generated, as an algebra, by two elements $g$ and $h$ subject to the following relations (see \cite{Ta}):
$$g^n=1,\quad\quad h^n=0,\quad\quad gh=qhg.$$
$H_n(q)$ is a Hopf algebra with the coalgebra structure and the antipode given by
$$
\begin{array}{lll}
\t(g)=g\otimes g, & \t(h)=h\otimes g+1\otimes h, & \e(g)=1,\\
\e(h)=0, & S(g)=g^{-1}=g^{n-1}, &  S(h)=-q^{-1}g^{n-1}h.
\end{array}
$$
Note that ${\rm dim}H_n(q)=n^2$, and $H_n(q)$ has a $k$-basis $\{g^ih^j|0\<i, j\<n-1\}$.
The Drinfeld double $D(H_n(q))$ can be described as follows.

Let $p\in k$. The Hopf algebra $H_n(p, q)$ is generated, as an algebra, by $a, b, c$ and $d$
subject to the following relations:
$$\begin{array}{lllll}
ba=qab,& db=qbd, & ca=qac,& dc=qcd,& bc=cb,\\
a^n=0, & b^n=1, &c^n=1,& d^n=0, & da-qad=p(1-bc).
\end{array}$$
The comultiplication, the counit and the antipode of $H_n(p, q)$ are determined by
$$
\begin{array}{lll}
\t(a)=a\otimes b+1\otimes a, & \e(a)=0, & S(a)=-ab^{-1}=-ab^{n-1},\\
\t(b)=b\otimes b, & \e(b)=1, & S(b)=b^{-1}=b^{n-1},\\
\t(c)=c\otimes c,& \e(c)=1, & S(c)=c^{-1}=c^{n-1},\\
\t(d)=d\otimes c+1\otimes d,& \e(d)=0, & S(d)=-dc^{-1}=-dc^{n-1}.
\end{array}
$$
$H_n(p, q)$ is $n^4$-dimensional with a $k$-basis $\{a^ib^jc^ld^k|0\<i, j, l, k\<n-1\}$.
$H_n(p, q)$ is not semisimple.
If $p\neq 0$, then $H_n(p, q)$ is isomorphic, as a Hopf algebra, to $D(H_n(q))$.
In particular, we have $H_n(p, q)\cong H_n(1, q)\cong D(H_n(q))$ for any $p\neq 0$.
For the details, the reader is directed to \cite{Ch1, Ch2}.
By \cite{Andrea, ZWLC}, one knows that the small quantum group
is isomorphic, as a Hopf algebra, to a quotient of $H_n(1,q)$.

\subsection{Indecomposable modules over $H_n(1,q)$}\selabel{1.3}
~~

Let $J:={\rm rad}(H_n(1,q))$ denote the Jacobson radical of $H_n(1,q)$.
Then  $J^3=0$ by \cite[Corollary 2.4]{Ch4}. That is, the Loewy length
of $H_n(1,q)$ is 3. In order to study the Green ring of $H_n(1,q)$,
we first need to give all finite dimensional indecomposable modules over $H_n(1,q)$.
We will use the notations of \cite{Ch4}. Unless otherwise stated, all modules are modules over $H_n(1,q)$
in what follows.

It follows from \cite{Ch4} that any indecomposable module has the same the socle series and the radical series.
For any $i, l\in\mathbb{Z}$ with $i\>0$, let $\a_i(l)=(i)_q(1-q^{i-l})$, where $(0)_q=0$.

Simple modules: $V(l,r)$, $1\<l\<n$, $r\in{\mathbb Z}_n$. $V(l,r)$ has
a standard $k$-basis $\{v_i|1\<i\<l\}$ such that
\begin{equation*}
\begin{array}{ll}
av_i=\left\{
\begin{array}{ll}
v_{i+1},& 1\<i<l,\\
0,& i=l,\\
\end{array}\right.&
dv_i=\left\{
\begin{array}{ll}
0, & i=1,\\
\a_{i-1}(l)v_{i-1},& 1<i\<l,\\
\end{array}\right.\\
bv_i=q^{r+i-1}v_i,\ 1\<i\<l, &cv_i=q^{i-r-l}v_i,\ 1\<i\<l.\\
\end{array}
\end{equation*}
The simple modules $V(n,r)$, $r\in{\mathbb Z}_n$, are both projective and injective.

Projective modules (not simple): Let $P(l, r)$ be the  projective cover of
$V(l, r)$, $1\<l<n$, $r\in{\mathbb Z}_n$. Then $P(l, r)$ is also the injective envelope of $V(l, r)$.
Moreover, we have (see \cite{Ch4})
$$\begin{array}{l}
{\rm soc}P(l, r)={\rm rad}^2P(l,r)\cong P(l, r)/{\rm rad}P(l, r)=P(l, r)/{\rm soc}^2P(l, r)\cong V(l, r),\\
{\rm soc}^2P(l, r)/{\rm soc}(P(l, r))={\rm rad}P(l, r)/{\rm rad}^2P(l, r)\cong 2V(n-l, r+l).\\
\end{array}$$

Non-simple and non-projective indecomposable modules can be divided into two
families: string modules and band modules.

String modules: The string modules are given by the syzygy functor $\O$ and cosyzygy functor $\O^{-1}$.
Let $1\<l<n$ and $r\in{\mathbb Z}_n$.
Then the minimal projective and injective resolutions of $V(l, r)$ are given by
$$
\cdots\ra4P(n-l, r+l)\ra3P(l, r)\ra2P(n-l, r+l)\ra P(l, r)\ra
V(l, r)\ra 0$$
and
$$
0\ra V(l, r)\ra P(l, r)\ra 2P(n-l, r+l)\ra 3P(l, r)\ra 4P(n-l, r+l)
\ra\cdots,$$
respectively. By these resolutions, one can describe the structure of $\O^sV(l, r)$ and $\O^{-s}V(l, r)$,
respectively, where $s\>1$ (see \cite{Ch4}).

Band modules: Let $\oo$ be a symbol with $\oo\not\in k$ and
let $\overline k=k\cup\{\oo\}$.
The band modules $M_s(l,r,\eta)$ can be described as follows, where $1\<l<n$, $r\in{\mathbb Z}_n$
and $\eta\in\ol{k}$ (see \cite{Ch4}).
By \cite[Lemma 3.7(1)]{Ch4}, the indecomposable module $M_1(l,r,\oo)$, $1\<l<n$, $r\in{\mathbb Z}_n$,
has a standard basis $\{v_1, v_2, \cdots, v_n\}$ such that
\begin{equation*}
\begin{array}{ll}
av_i=\left\{\begin{array}{ll}
0, & i=n-l \mbox{ or }n,\\
v_{i+1}, & \mbox{ otherwise },\\
\end{array}\right. &
dv_i=\left\{\begin{array}{ll}
v_n, & i=1,\\
\a_{i-1}(n-l)v_{i-1}, & 1<i\<n,\\
\end{array}\right.\\
bv_i=q^{r+l+i-1}v_i, & cv_i=q^{i-r}v_i.\\
\end{array}
\end{equation*}
By \cite[Lemma 3.7(2)]{Ch4}, the indecomposable module $M_1(l,r,\eta)$, $1\< l<n$, $r\in{\mathbb Z}_n$, $\eta\in k$,
has a standard basis $\{v_1, v_2, \cdots, v_n\}$ with the action given by
\begin{equation*}
\begin{array}{ll}
av_i=\left\{\begin{array}{ll}
v_{i+1} , & 1\<i<n,\\
0 , & i=n,\\
\end{array}\right. &
dv_i=\left\{\begin{array}{ll}
\eta q^lv_n, & i=1,\\
\a_{i-1}(n-l)v_{i-1}, & 1<i\<n,\\
\end{array}\right.\\
bv_i=q^{r+l+i-1}v_i, & cv_i=q^{i-r}v_i.\\
\end{array}
\end{equation*}
Then the band modules $M_s(l, r,\eta)$ are determined recursively by the almost split sequences
$$0\ra M_s(l, r,\eta)\ra M_{s-1}(l, r,\eta)\oplus M_{s+1}(l, r,\eta)\ra M_s(l, r,\eta)\ra 0,$$
where $s\>1$, $M_0(l, r,\eta)=0$, $1\<l<n$, $r\in{\mathbb Z}_n$ and $\eta\in\ol{k}$
(see \cite{Ch3, Ch4}). $M_s(l, r,\eta)$ is a submodule of $sP(l,r)$
and a quotient module of $sP(n-l, r+l)$, and there is an exact sequence
$$0\ra M_s(l, r,\eta)\hookrightarrow sP(l, r)\ra M_s(n-l, r+l, -\eta q^l)\ra 0.$$
Hence $\O M_s(l, r, \eta)\cong \O^{-1} M_s(l, r, \eta)\cong M_s(n-l, r+l, -\eta q^l)$.
Moreover, for any $1\<i<s$, there is an exact sequence of modules
$$0\ra M_i(l, r,\eta)\hookrightarrow M_s(l, r, \eta)\ra M_{s-i}(l, r, \eta)\ra 0.$$

Throughout the following, let $n$ be a fixed positive integer with $n>2$,
and $q\in k$ a root of unity of order $n$.
Let $P(n, r)=V(n, r)$ and $\O^0V(l,r)=V(l,r)$ for all $1\<l<n$ and $r\in\mathbb{Z}_n$,
and let $\a\oo=\oo\a=\oo$ for any $0\neq\a\in k$.
For any $t\in\mathbb Z$, let $c(t):=[\frac{t+1}{2}]$ be the integer part of $\frac{t+1}{2}$.
That is, $c(t)$ is the maximal integer with respect to $c(t)\<\frac{t+1}{2}$.
Then $c(t)+c(t-1)=t$.

\section{\bf The Green Ring of $H_n(1,q)$}\selabel{3}

In this section, we study the Green ring  $r(H_n(1,q))$ of $H_n(1,q)$, which is a commutative ring
since $H_n(1,q)$ is quasitriangular.

The projective class ring $r_p(H_n(1,q))$ was described in
\cite{ChHasLinSun}, which is isomorphic to the quotient ring of the polynomial ring
$\mathbb{Z}[x,y]$ modulo the ideal $I$ generated by $x^n-1$ and
$$\begin{array}{c}
(\sum_{i=0}^{c(n-2)}(-1)^i\binom{n-1-i}{i}x^iy^{n-1-2i})
(\sum_{i=0}^{c(n-1)}(-1)^i\frac{n}{n-i}\binom{n-i}{i}x^iy^{n-2i}-2).\\
\end{array}$$
An isomorphism $\mathbb{Z}[x, y]/I\ra r_p(H_n(1,q))$ is given by
$x+I\mapsto[V(1,1)]$ and $y+I\mapsto[V(2,0)]$ (see \cite[Theorem 5.9]{ChHasLinSun}).
By \cite[Corollary 5.2]{ChHasLinSun}, $r_p(H_n(1,q))$ is a free $\mathbb Z$-module with a $\mathbb Z$-basis
$$\{[V(s,r)], [P(l,r)]|1\<s\<n, 1\<l<n, r\in\mathbb{Z}_n\}.$$

Let $x=[V(1,1)]$, $y=[V(2,0)]$, $z_+=[\O V(1,0)]$ and $z_{-}=[\O^{-1}V(1,0)]$ in $r(H_n(1,q))$.
Let $R$ be the subring of $r(H_n(1,q))$ generated by $x, y, z_+$ and $z_-$. Then
$r_p(H_n(1,q))\subset R$ by \cite[Proposition 5.5]{ChHasLinSun}.

\begin{proposition}\label{3.1}
$R$ is a free $\mathbb Z$-module with a $\mathbb Z$-basis
$$\{[V(s,r)], [P(l,r)], [\O^{\pm m}V(l,r)]|1\leqslant s\leqslant n,
1\leqslant l\leqslant n-1, m\>1, r\in\mathbb{Z}_n\}.$$
\end{proposition}

\begin{proof}
Let $R'$ be the $\mathbb Z$-submodule of $r(H_n(1,q))$ generated by the set
$$\{[V(s,r)], [P(l,r)], [\O^{\pm m}V(l,r)]|1\leqslant s\leqslant n,
1\leqslant l\leqslant n-1, m\>1, r\in\mathbb{Z}_n\}.$$
Since the tensor product of a projective module with any module is projective,
it follows from \cite[Corollaries 3.2, 5.3 and 5.5, Propositions 3.6, 3.7, 5.2 and 5.4]{ChHasSun}
that $R'$ is a subring of $r(H_n(1,q))$.
Hence $R\subseteq R'$. It is left to show $R'\subseteq R$.
By $r_p(H_n(1,q))\subset R$ and \cite[Corollary 5.2]{ChHasLinSun}, it is enough to show
$[\O^{\pm m}V(l,r)]\in R$ for all $1\leqslant l\leqslant n-1$, $m\>1$ and $r\in\mathbb{Z}_n$.

At first, $[\O V(1,0)]=z_+\in R$. Let $m\>1$ and assume that $[\O^mV(1,0)]\in R$.
Then by \cite[Proposition 5.2]{ChHasSun}, $\O V(1,0)\ot\O^mV(1,0)\cong\O^{m+1}V(1,0)\oplus P$
for some projective module $P$. Hence $[\O^{m+1}V(1,0)]=z_+[\O^mV(1,0)]-[P]\in R$
by $[P]\in r_p(H_n(1,q))\subset R$. This shows that $[\O^mV(1,0)]\in R$ for all
$m\>1$. Similarly, one can show that $[\O^{-m}V(1,0)]\in R$ for all
$m\>1$. Then for any $1\<l\<n-1$, $m\>1$ and $r\in\mathbb{Z}_n$, by \cite[Proposition 3.6]{ChHasSun}, we have
$V(l, r)\ot\O^{\pm m}V(1,0)\cong\O^{\pm m}V(l,r)\oplus P$ for some projective module $P$.
By $[P], [V(l,r)]\in r_p(H_n(1,q))\subset R$, it follows that
$[\O^{\pm m}V(l,r)]=[V(l,r)][\O^{\pm m}V(1,0)]-[P]\in R$.
This completes the proof.
\end{proof}

\begin{lemma}\label{3.2}
Let $l, m\in\mathbb{Z}$ with $1\<l\<\frac{m-1}{2}$. Then for any $s\>0$,
$$\begin{array}{c}
\sum_{i=0}^{s}(-1)^i\frac{m-2l+2i}{m-2l+i}\binom{m-2l+i}{i}=(-1)^s\binom{m-2l+s}{s}.\\
\end{array}$$
In particular, $\sum_{i=0}^{l-1}(-1)^i\frac{m-2l+2i}{m-2l+i}\binom{m-2l+i}{i}=(-1)^{l-1}\binom{m-l-1}{l-1}$.
\end{lemma}

\begin{proof}
It can be shown by induction $s$.
\end{proof}

\begin{lemma}\label{3.3}
The following relations are satisfied in $r(H_n(1,q))$:\\
$(1)$ $\sum_{i=0}^{c(n-2)}[P(2i+1,-i)]=x[V(n,0)]^2$.\\
$(2)$ $\sum_{i=1}^{c(n-2)}[P(2i+1,-i)]=\sum_{i=1}^{c(n-2)}(-1)^{i-1}\binom{n-i-2}{i-1}x^{i+1}y^{n-1-2i}[V(n,0)].$\\
$(3)$ $\sum_{i=1}^{c(n-1)}[P(2i,-i)]=\sum_{i=1}^{c(n-1)}(-1)^{i-1}\binom{n-i-1}{i-1}x^iy^{n-2i}[V(n,0)].$
\end{lemma}

\begin{proof}
(1) By \cite[Corollary 5.4(1)]{ChHasLinSun}, we have $x[V(n,0)]^2=[V(n,1)][V(n,0)]=[V(n,1)\ot V(n,0)]$.
Then by \cite[Proposition 3.1(2)]{ChHasSun}, one gets that
$$\begin{array}{c}
V(n,1)\ot V(n,0)\cong\oplus_{i=c(n-1)}^{n-1}P(2n-1-2i, i+1)
\cong\oplus_{i=0}^{c(n-2)}P(2i+1, -i).\\
\end{array}$$
Therefore, $\sum_{i=0}^{c(n-2)}[P(2i+1,-i)]=x[V(n,0)]^2$.

(2) By (1) and \cite[Lemma 5.6]{ChHasLinSun}, we have
$$\begin{array}{rl}
&\sum_{i=1}^{c(n-2)}[P(2i+1,-i)]=x[V(n,0)]^2-[P(1,0)]\\
=&(\sum_{i=0}^{[\frac{n-1}{2}]}(-1)^i\binom{n-1-i}{i}x^{i+1}y^{n-1-2i}\\
&-\sum_{i=0}^{[\frac{n-1}{2}]}(-1)^i\frac{n-1}{n-1-i}\binom{n-1-i}{i}x^{i+1}y^{n-1-2i})[V(n,0)]\\
=&\sum_{i=1}^{[\frac{n-1}{2}]}(-1)^{i-1}(\frac{n-1}{n-1-i}\binom{n-1-i}{i}-\binom{n-1-i}{i})x^{i+1}y^{n-1-2i}[V(n,0)]\\
=&\sum_{i=1}^{c(n-2)}(-1)^{i-1}\binom{n-i-2}{i-1}x^{i+1}y^{n-1-2i}[V(n,0)].\\
\end{array}$$

(3) We may assume that $n$ is even since the proof is similar for $n$ to be odd.
In this case, by \cite[Corollary 5.4(1)-(2), Lemma 5.6(2)]{ChHasLinSun} and Lemma \ref{3.2}, we have
$$\begin{array}{rl}
&\sum_{i=1}^{c(n-1)}[P(2i, -i)]=\sum_{i=1}^{\frac{n-2}{2}}x^{-i}[P(2i,0)]+x^{-\frac{n}{2}}[V(n,0)]\\
=&\sum_{i=1}^{\frac{n-2}{2}}\sum_{j=0}^{\frac{n-2i}{2}}(-1)^j\frac{n-2i}{n-2i-j}
\binom{n-2i-j}{j}x^{i+j}y^{n-2i-2j}[V(n,0)]+x^{\frac{n}{2}}[V(n,0)]\\
=&(\sum_{l=1}^{\frac{n-2}{2}}\sum_{j=0}^{l-1}(-1)^j\frac{n-2l+2j}{n-2l+j}
\binom{n-2l+j}{j}x^{l}y^{n-2l}+\sum_{j=1}^{\frac{n-2}{2}}(-1)^j2x^{\frac{n}{2}}+x^{\frac{n}{2}})[V(n,0)]\\
=&(\sum_{l=1}^{\frac{n-2}{2}}(-1)^{l-1}\binom{n-l-1}{l-1}x^{l}y^{n-2l}
+(2\sum_{j=1}^{\frac{n-2}{2}}(-1)^j+1)x^{\frac{n}{2}})[V(n,0)]\\
=&\sum_{l=1}^{\frac{n}{2}}(-1)^{l-1}\binom{n-l-1}{l-1}x^{l}y^{n-2l}[V(n,0)].\\
\end{array}$$
\end{proof}

\begin{lemma}\label{3.4}
The following relations are satisfied in $r(H_n(1,q))$:\\
$(1)$ $z_+z_-=1+(2y+4\sum_{i=1}^{c(n-2)}(-1)^{i-1}\binom{n-i-2}{i-1}x^{i+1}y^{n-1-2i})[V(n,0)]$.\\
$(2)$ $z_{+}[V(n,0)]=z_{-}[V(n,0)]=(1+2\sum_{i=1}^{c(n-1)}(-1)^{i-1}\binom{n-i-1}{i-1}x^iy^{n-2i})[V(n,0)]$.
\end{lemma}

\begin{proof}
(1) By \cite[Proposition 5.2(1)]{ChHasSun}, we have
$$\begin{array}{c}
\O V(1,0)\ot\O^{-1}V(1,0)
\cong V(1,0)\oplus 2P(n-1,1)
\oplus(\oplus_{i=1}^{c(n-2)}4P(2i+1,-i)).\\
\end{array}$$
Hence by \cite[Corollary 5.4(2)-(3)]{ChHasLinSun} and Lemma \ref{3.3}(2), we have
$$\begin{array}{rl}
z_+z_-=&[V(1,0)]+2x[P(n-1,0)]+4\sum_{i=1}^{c(n-2)}[P(2i+1,-i)]\\
=&1+2y[V(n,0)]+4\sum_{i=1}^{c(n-2)}(-1)^{i-1}\binom{n-i-2}{i-1}x^{i+1}y^{n-1-2i}[V(n,0)]\\
=&1+(2y+4\sum_{i=1}^{c(n-2)}(-1)^{i-1}\binom{n-i-2}{i-1}x^{i+1}y^{n-1-2i})[V(n,0)].\\
\end{array}$$

(2) By \cite[Proposition 3.7]{ChHasSun}, we have
$$\begin{array}{rl}
\O^{\pm1}V(1,0)\ot V(n,0)\cong&V(n,0)\oplus(\oplus_{i=c(n)}^{n-1}2P(2n-2i,i))\\
\cong&V(n,0)\oplus(\oplus_{i=1}^{c(n-1)}2P(2i,-i)).\\
\end{array}$$
Then it follows from Lemma \ref{3.3}(3) that
$$\begin{array}{rl}
z_{\pm}[V(n,0)]=&[V(n,0)]+2\sum_{i=1}^{c(n-1)}[P(2i,-i)]\\
=&[V(n,0)]+2\sum_{i=1}^{c(n-1)}(-1)^{i-1}\binom{n-i-1}{i-1}x^iy^{n-2i}[V(n,0)]\\
=&(1+2\sum_{i=1}^{c(n-1)}(-1)^{i-1}\binom{n-i-1}{i-1}x^iy^{n-2i})[V(n,0)].\\
\end{array}$$
\end{proof}

\begin{proposition}\label{3.5}
The following set is also a $\mathbb Z$-basis of $R$:
$$\{x^iy^j, x^iy^lz_+^m, x^iy^lz_-^m|0\<i\<n-1, 0\<j\<2n-2, 0\<l\<n-2, m\>1\}.$$
\end{proposition}

\begin{proof}
Let $N$ be the $\mathbb{Z}$-submodule of $R$ generated by
$$\{[\O^{\pm m}V(l,r)]|1\<l\<n-1, r\in{\mathbb Z}_n, m\>1\}.$$
Then by \cite[Corollary 5.2]{ChHasLinSun} and Proposition \ref{3.1}, we have
$R=r_p(H_n(1,q))\oplus N$ as $\mathbb Z$-modules.
Let $\ol{R}=R/r_p(H_n(1,q))$ be the quotient $\mathbb Z$-module,
and $\pi: R\ra \ol{R}$ the corresponding canonical epimorphism.
Then $\{\pi([\O^{\pm m}V(l,r)])|1\<l\<n-1, r\in{\mathbb Z}_n, m\>1\}$
is a $\mathbb Z$-basis of $\ol{R}$. By \cite[Corollary 5.8]{ChHasLinSun}, it is enough to show that
$\{\pi(x^iy^lz_{\pm}^m)|0\<i\<n-1, 0\<l\<n-2, m\>1\}$ is also a
$\mathbb Z$-basis of $\ol{R}$.
For a subset $S$ of $\ol R$, let $\langle S\rangle$ denote the $\mathbb Z$-submodule
of $\ol R$ generated by $S$.

We first show by induction on $m$ that $(\O V(1,0))^{\ot m}\cong \O^mV(1,0)\oplus P_m^+$
for some projective module $P_m^+$, where $m\>1$. For $m=1$, it is trivial.
Now let $m\>1$ and assume $(\O V(1,0))^{\ot m}\cong \O^mV(1,0)\oplus P_m^+$
for some projective module $P_m^+$. Then by \cite[Proposition 5.2]{ChHasSun}, we have
$$\begin{array}{rl}
(\O V(1,0))^{\ot(m+1)}\cong&\O V(1,0)\ot \O^mV(1,0)\oplus\O V(1,0)\ot P_m^+\\
\cong&\O^{m+1}V(1,0)\oplus Q\oplus\O V(1,0)\ot P_m^+\\
\end{array}$$
for some projective module $Q$. Note that $P_{m+1}^+=Q\oplus\O V(1,0)\ot P_m^+$ is projective
since $\O V(1,0)\ot P_m^+$ is projective. Thus,
for any $m\>1$, $(\O V(1,0))^{\ot m}\cong \O^mV(1,0)\oplus P_m^+$
for some projective module $P_m^+$. Similarly, one can show that
$(\O^{-1}V(1,0))^{\ot m}\cong \O^{-m}V(1,0)\oplus P_m^-$
for some projective module $P_m^-$, where $m\>1$.

Next, let $1\<l\<n-1$, $r\in\mathbb{Z}_n$ and $m\>1$. Then by \cite[Proposition 3.6]{ChHasSun}, we have
$$V(l,r)\ot\O^{\pm m}V(1,0)\cong\O^{\pm m}V(l,r)\oplus P$$
for some projective module $P$. Hence
$$\begin{array}{rl}
V(l,r)\ot(\O^{\pm 1}V(1,0))^{\ot m}\cong&V(l,r)\ot\O^{\pm m}V(1,0)\oplus  V(l,r)\ot P_m^{\pm}\\
\cong&\O^{\pm m}V(l,r)\oplus P_{l,r,m}^{\pm}\\
\end{array}$$
for some projective module $P^{\pm}_{l,r,m}$.
Thus, by \cite[Corollary 5.4(1) and Lemma 5.6(1)]{ChHasLinSun},
we have
$$\begin{array}{rl}
\pi([\O^{\pm m}V(l,r)])=&\pi(x^r[V(l,0)]z_{\pm}^m)\\
=&\pi(\sum_{i=0}^{[\frac{l-1}{2}]}(-1)^i\binom{l-1-i}{i}x^{r+i}y^{l-1-2i}z_{\pm}^m)\\
=&\sum_{i=0}^{[\frac{l-1}{2}]}(-1)^i\binom{l-1-i}{i}\pi(x^{r+i}y^{l-1-2i}z_{\pm}^m).\\
\end{array}$$
It follows that $\ol{R}$ is generated, as a $\mathbb Z$-module, by
$\{\pi(x^iy^lz_{\pm}^m)|0\<i\<n-1, 0\<l\<n-2, m\>1\}$.

Finally, let $0\<l\<n-2$ and $m\>1$. Let us regard $V(2,0)^{\ot 0}=V(1,0)$ and $\binom{0}{0}=1$
for simplicity. Then by \cite[Lemma 5.3]{ChHasLinSun} and the discussion above, we have
$$\begin{array}{rl}
&V(2,0)^{\ot l}\ot(\O V(1,0))^{\ot m}\\
\cong&\oplus_{j=0}^{[\frac{l}{2}]}
\frac{l-2j+1}{l-j+1}\binom{l}{j}V(l+1-2j, j)\ot(\O V(1,0))^{\ot m}\\
\cong&\oplus_{j=0}^{[\frac{l}{2}]}
\frac{l-2j+1}{l-j+1}\binom{l}{j}(\O^mV(l+1-2j, j)\oplus P^+_{l+1-2j, j,m})\\
\cong&(\oplus_{j=0}^{[\frac{l}{2}]}
\frac{l-2j+1}{l-j+1}\binom{l}{j}\O^mV(l+1-2j, j))\oplus P\\
\end{array}$$
for some projective module $P$.
Thus, for any $0\<i\<n-1$, $0\<l\<n-2$ and $m\>1$, it follows from \cite[Proposition 3.6]{ChHasSun}
and \cite[Corollary 5.4(1)]{ChHasLinSun} that
$$\begin{array}{rl}
\pi(x^iy^lz_+^m)=&\pi(\sum_{j=0}^{[\frac{l}{2}]}
\frac{l-2j+1}{l-j+1}\binom{l}{j}[\O^mV(l+1-2j, j+i)])\\
=&\pi([\O^mV(l+1, i)])+\sum\limits_{1\<j\<[\frac{l}{2}]}
\frac{l-2j+1}{l-j+1}\binom{l}{j}\pi([\O^mV(l+1-2j, j+i)]).\\
\end{array}$$
It follows that the set $\{\pi(x^iy^lz_+^m)|0\<i\<n-1, 0\<l\<n-2, m\>1\}$
is linearly independent over $\mathbb Z$.
Similarly, one can show that the set $\{\pi(x^iy^lz_-^m)|0\<i\<n-1, 0\<l\<n-2, m\>1\}$
is linearly independent over $\mathbb Z$ too.
Moreover, we have
$$\begin{array}{rl}
&\langle\{\pi([\O^mV(l,r)])|1\<l\<n-1, r\in\mathbb{Z}_n, m\>1\}\rangle\\
=&\langle\{\pi(x^iy^lz_+^m)|0\<i\<n-1, 0\<l\<n-2, m\>1\}\rangle\\
\end{array}$$
and
$$\begin{array}{rl}
&\langle\{\pi([\O^{-m}V(l,r)])|1\<l\<n-1, r\in\mathbb{Z}_n, m\>1\}\rangle\\
=&\langle\{\pi(x^iy^lz_-^m)|0\<i\<n-1, 0\<l\<n-2, m\>1\}\rangle.\\
\end{array}$$
Therefore, $\{\pi(x^iy^lz_{\pm}^m)|0\<i\<n-1, 0\<l\<n-2, m\>1\}$
is a $\mathbb Z$-basis of $\ol R$. This completes the proof.
\end{proof}

Now let $w_{m,\eta}=[M_m(1,0,\eta)]$ in $r(H_n(1,q))$ for any $m\>1$ and $\eta\in\ol{k}$.

\begin{lemma}\label{3.6}
Let $m, s\>1$ and $\eta, \a\in\ol{k}$. Then we have the following relations in $r(H_n(1,q))$:\\
$(1)$ $w_{m,\eta}[V(n,0)]=m(1+\sum_{i=1}^{c(n-1)}(-1)^{i-1}\binom{n-1-i}{i-1}x^iy^{n-2i})[V(n,0)]$.\\
$(2)$ $z_+w_{m,\eta}=(\sum_{i=1}^{c(n-1)}(-1)^{i-1}\binom{n-1-i}{i-1}x^{i}y^{n-2i})w_{m,\eta}+mx[V(n,0)]^2$.\\
$(3)$ $z_-w_{m,\eta}=(\sum_{i=1}^{c(n-1)}(-1)^{i-1}\binom{n-i-1}{i-1}x^{i}y^{n-2i})w_{m,\eta}$\\
\mbox{\hspace{1.7cm}}$+m(y+\sum_{i=1}^{c(n-2)}(-1)^{i-1}\binom{n-i-2}{i-1}x^{i+1}y^{n-1-2i})[V(n,0)]$.\\
$(4)$ If $\eta\neq\a$, then $w_{m,\eta}w_{s,\a}=msx[V(n,0)]^2$.\\
$(5)$ If $m\<s$, then
$$\begin{array}{c}
w_{m,\eta}w_{s,\eta}=w_{m,\eta}(1+\sum_{i=1}^{c(n-1)}(-1)^{i-1}\binom{n-i-1}{i-1}x^iy^{n-2i})
+(s-1)mx[V(n,0)]^2.\\
\end{array}$$
\end{lemma}

\begin{proof}
(1) By \cite[Theorem 3.17]{ChHasSun}, we have
$$\begin{array}{rl}
M_m(1,0,\eta)\ot V(n,0)\cong&mP(n,0)\oplus(\oplus_{i=c(n)}^{n-1}mP(2n-2i,i))\\
\cong&mV(n,0)\oplus(\oplus_{i=1}^{c(n-1)}mP(2i,-i)).\\
\end{array}$$
Then by Lemma \ref{3.3}(3), one gets that
$$\begin{array}{rl}
w_{m,\eta}[V(n,0)]=&m[V(n,0)]+\sum_{i=1}^{c(n-1)}m[P(2i,-i)]\\
=&m(1+\sum_{i=1}^{c(n-1)}(-1)^{i-1}\binom{n-i-1}{i-1}x^iy^{n-2i})[V(n,0)].\\
\end{array}$$

(2) By \cite[Theorem 3.16]{ChHasSun}, we have
$$\begin{array}{rl}
&V(n-1,1)\ot M_m(1,0,\eta)\\
\cong&M_m(n-1,1,-\eta q)\oplus(\oplus_{i=c(n-1)}^{n-2}mP(2n-1-2i, i+1))\\
\cong&M_m(n-1,1,-\eta q)\oplus(\oplus_{i=1}^{c(n-2)}mP(2i+1, -i)).\\
\end{array}$$
Then by \cite[Proposition 5.6(1)]{ChHasSun}, one gets that
$$\begin{array}{rl}
&\O V(1,0)\ot M_m(1,0,\eta)\\
\cong&M_m(n-1,1,-\eta q)\oplus mP(1,0)\oplus(\oplus_{i=1}^{c(n-2)}2mP(2i+1, -i))\\
\cong&V(n-1,1)\ot M_m(1,0,\eta)\oplus(\oplus_{i=0}^{c(n-2)}mP(2i+1, -i)).\\
\end{array}$$
Thus, it follows from \cite[Corollary 5.4(1), Lemma 5.6(1)]{ChHasLinSun} and Lemma \ref{3.3}(1) that
$$\begin{array}{rl}
z_+w_{m,\eta}=&x[V(n-1,0)]w_{m,\eta}+\sum_{i=0}^{c(n-2)}m[P(2i+1, -i)]\\
=&(\sum_{i=0}^{[\frac{n-2}{2}]}(-1)^i\binom{n-2-i}{i}x^{i+1}y^{n-2-2i})w_{m,\eta}+mx[V(n,0)]^2\\
=&(\sum_{i=1}^{c(n-1)}(-1)^{i-1}\binom{n-1-i}{i-1}x^{i}y^{n-2i})w_{m,\eta}+mx[V(n,0)]^2.\\
\end{array}$$

(3) By \cite[Corollary 5.7(1)]{ChHasSun} and the proof of (2), we have
$$\begin{array}{rl}
&\O^{-1}V(1,0)\ot M_m(1,0,\eta)\\
\cong&M_m(n-1,1,-\eta q)\oplus mP(n-1,1)\oplus(\oplus_{i=1}^{c(n-2)}2mP(2i+1, -i))\\
\cong&V(n-1,1)\ot M_m(1,0,\eta)\oplus mP(n-1,1)\oplus(\oplus_{i=1}^{c(n-2)}mP(2i+1, -i)).\\
\end{array}$$
Then by \cite[Corollary 5.4(1)-(3), Lemma 5.6(1)]{ChHasLinSun} and Lemma \ref{3.3}(2), one gets that
$$\begin{array}{rl}
z_-w_{m,\eta}=&x[V(n-1,0)]w_{m,\eta}+mx[P(n-1,0)]+\sum_{i=1}^{c(n-2)}m[P(2i+1, -i)]\\
=&x[V(n-1,0)]w_{m,\eta}+my[V(n,0)]\\
&+m\sum_{i=1}^{c(n-2)}(-1)^{i-1}\binom{n-i-2}{i-1}x^{i+1}y^{n-1-2i}[V(n,0)]\\
=&(\sum_{i=1}^{c(n-1)}(-1)^{i-1}\binom{n-1-i}{i-1}x^{i}y^{n-2i})w_{m,\eta}\\
&+m(y+\sum_{i=1}^{c(n-2)}(-1)^{i-1}\binom{n-i-2}{i-1}x^{i+1}y^{n-1-2i})[V(n,0)].\\
\end{array}$$

(4) Assume that $\eta\neq\a$. Then by \cite[Lemma 5.15]{ChHasSun}, we have
$$M_m(1,0,\eta)\ot M_s(1,0,\a)\cong\oplus_{i=1}^{c(n)}msP(2i-1,1-i)
\cong\oplus_{i=0}^{c(n-2)}msP(2i+1,-i).$$
Thus, it follows from Lemma \ref{3.3}(1) that $w_{m,\eta}w_{s,\a}=msx[V(n,0)]^2$.

(5) Assume that $m\<s$. Then by \cite[Lemma 5.21]{ChHasSun} and the proof of (2), one has
$$\begin{array}{rl}
&M_m(1,0,\eta)\ot M_s(1,0,\eta)\\
\cong&M_m(1,0,\eta)\oplus M_m(n-1,1,-\eta q)\oplus(s-1)mP(1,0)\oplus(\oplus_{i=1}^{c(n-2)}msP(2i+1, -i))\\
\cong&M_m(1,0,\eta)\oplus V(n-1,1)\ot M_m(1,0,\eta)\oplus(\oplus_{i=0}^{c(n-2)}(s-1)mP(2i+1, -i)).\\
\end{array}$$
Thus, it follows from \cite[Corollary 5.4(1), Lemma 5.6(1)]{ChHasLinSun} and Lemma \ref{3.3}(1) that
$$\begin{array}{rl}
w_{m,\eta}w_{s,\eta}=&w_{m,\eta}+x[V(n-1,0)]w_{m,\eta}+\sum_{i=0}^{c(n-2)}(s-1)m[P(2i+1, -i)]\\
=&(1+\sum_{i=1}^{c(n-1)}(-1)^{i-1}\binom{n-1-i}{i-1}x^{i}y^{n-2i})w_{m,\eta}+(s-1)mx[V(n,0)]^2.\\
\end{array}$$
\end{proof}

\begin{proposition}\label{3.7}
$r(H_n(1,q))$ is generated, as a ring, by $$\{x, y, z_+, z_-, w_{m,\eta}|m\>1, \eta\in\ol{k}\}.$$
\end{proposition}

\begin{proof}
Let $R'$ be the subring of $r(H_n(1,q))$ generated by $\{x,y, z_+, z_-, w_{m,\eta}|m\>1, \eta\in\ol{k}\}$.
Then $R\subseteq R'$. By Proposition \ref{3.1} and the classification of finite dimensional modules over $H_n(1,q)$
(see \cite{Ch4} or \seref{2}), it is enough to show that $[M_m(l,r,\eta)]\in R'$ for all
$1\<l\<n-1$, $m\>1$, $r\in\mathbb{Z}_n$ and $\eta\in\ol{k}$.
In fact, by \cite[Theorem 3.16]{ChHasSun}, we have
$$V(l,r)\ot M_m(1,0,\eta)\cong M_m(l,r,\eta q^{1-l}(l)_q)\oplus(\oplus_{c(l)\<i\<l-1}mP(n+l-2i,r+i)).$$
Then it follows from Proposition \ref{3.1} that
$$\begin{array}{c}
[M_m(l,r,\eta  q^{1-l}(l)_q)]=[V(l,r)]w_{m,\eta}-\sum_{c(l)\<i\<l-1}m[P(n+l-2i,r+i)]\in R'\\
\end{array}$$
for all $1\<l\<n-1$, $m\>1$, $r\in\mathbb{Z}_n$ and $\eta\in\ol{k}$.
Thus, the proposition follows from the fact that the map $\ol{k}\ra\ol{k}$,
$\eta\mapsto\eta q^{1-l}(l)_q$, is a bijection for any fixed $1\<l\<n-1$.
\end{proof}

\begin{proposition}\label{3.8}
The following set is a $\mathbb Z$-basis of $r(H_n(1,q))$:
$$\left\{x^iy^j, x^iy^lz_+^m, x^iy^lz_-^m, x^iy^lw_{m,\eta}
\left|\begin{array}{l}
0\<i\<n-1, 0\<j\<2n-2,\\
0\<l\<n-2, m\>1, \eta\in\ol{k}\\
\end{array}\right\}\right..$$
\end{proposition}

\begin{proof}
It is similar to Proposition \ref{3.5}. Let $R_0$ be the $\mathbb Z$-submodule of $r(H_n(1,q))$
generated by $\{[M_m(l,r,\eta)]|m\>1, 1\<l\<n-1, r\in{\mathbb Z}_n, \eta\in\ol{k}\}$.
Then by Proposition \ref{3.1} and the classification of finite dimensional modules over $H_n(1,q)$
(see \cite{Ch4} or \seref{2}), $r(H_n(1,q))=R\oplus R_0$ as $\mathbb Z$-modules.
Let $r(H_n(1,q))/R$ be the quotient $\mathbb Z$-module of $r(H_n(1,q))$ modulo $R$,
and $\pi: r(H_n(1,q))\ra r(H_n(1,q))/R$ the corresponding canonical projection.
Then $r(H_n(1,q))/R$ is a free $\mathbb Z$-module with a $\mathbb Z$-basis given by
$$\{\pi([M_m(l,r,\eta q^{1-l}(l)_q)])|m\>1, 1\<l\<n-1, r\in{\mathbb Z}_n, \eta\in\ol{k}\}.$$

Let $m\>1$, $1\<l\<n-1$, $r\in{\mathbb Z}_n$ and $\eta\in\ol{k}$.
Then by \cite[Corollary 5.4(1), Lemma 5.6(1)]{ChHasLinSun} and the proof of Proposition \ref{3.7}, we have
$$\begin{array}{rl}
\pi([M_m(l,r,\eta  q^{1-l}(l)_q)])=&\pi([V(l,r)]w_{m,\eta})\\
=&\sum_{i=0}^{[\frac{l-1}{2}]}(-1)^i\binom{l-1-i}{i}\pi(x^{r+i}y^{l-1-2i}w_{m,\eta}).\\
\end{array}$$
It follows that $r(H_n(1,q))/R$ is generated, as a $\mathbb Z$-module, by
$$\{\pi(x^iy^lw_{m,\eta})|0\<i\<n-1, 0\<l\<n-2,  m\>1, \eta\in\ol{k}\}.$$

Now let $0\<i\<n-1$, $0\<l\<n-2$, $m\>1$ and $\eta\in\ol{k}$. Then by \cite[Lemma 5.3]{ChHasLinSun}
and \cite[Theorem 3.16]{ChHasSun}, we have
$$\begin{array}{rl}
&V(2,0)^{\ot l}\ot M_m(1,0, \eta)\\
\cong&\oplus_{j=0}^{[\frac{l}{2}]}
\frac{l-2j+1}{l-j+1}\binom{l}{j}V(l+1-2j, j)\ot M_m(1,0, \eta)\\
\cong&\oplus_{j=0}^{[\frac{l}{2}]}
\frac{l-2j+1}{l-j+1}\binom{l}{j}M_m(l+1-2j, j, \eta q^{2j-l}(l+1-2j)_q)\oplus P\\
\end{array}$$
for some projective module $P$.
Thus, by \cite[Corollary 5.4(1)]{ChHasLinSun} and \cite[Lemma 3.8]{ChHasSun}, one gets
$$\begin{array}{rl}
\pi(x^iy^lw_{m,\eta})=&\pi(\sum_{j=0}^{[\frac{l}{2}]}
\frac{l-2j+1}{l-j+1}\binom{l}{j}[M_m(l+1-2j, i+j, \eta q^{2j-l}(l+1-2j)_q)])\\
=&\pi([M_m(l+1, i, \eta q^{-l}(l+1)_q)])\\
&+\sum\limits_{1\<j\<[\frac{l}{2}]}
\frac{l-2j+1}{l-j+1}\binom{l}{j}\pi([M_m(l+1-2j, i+j, \eta q^{2j-l}(l+1-2j)_q)]).\\
\end{array}$$
It follows that the set $\{\pi(x^iy^lw_{m,\eta})|0\<i\<n-1, 0\<l\<n-2, m\>1, \eta\in\ol{k}\}$
is linearly independent over $\mathbb Z$, and so it is a $\mathbb Z$-basis of $r(H_n(1,q))/R$.
Thus, the proposition follows from Proposition \ref{3.5}.
\end{proof}

Let $\mathbb{Z}[x,y,z_+,z_-]$ be the polynomial ring in variables $x, y, z_+, z_-$.
Define the polynomials $f_1(x,y)$, $f_2(x,y)$, $f_3(x,y)$ and $f_4(x,y)$ in $\mathbb{Z}[x,y]\subset\mathbb{Z}[x,y,z_+,z_-]$ by
$$\begin{array}{l}
f_1(x,y)=\sum_{i=0}^{c(n-2)}(-1)^i\binom{n-1-i}{i}x^iy^{n-1-2i},\\
f_2(x,y)=\sum_{i=0}^{c(n-1)}(-1)^i\frac{n}{n-i}\binom{n-i}{i}x^iy^{n-2i}-2,\\
f_3(x,y)=\sum_{i=1}^{c(n-2)}(-1)^{i-1}\binom{n-i-2}{i-1}x^{i+1}y^{n-1-2i},\\
f_4(x,y)=\sum_{i=1}^{c(n-1)}(-1)^{i-1}\binom{n-i-1}{i-1}x^iy^{n-2i}.\\
\end{array}$$
Let $I$ be the ideal of $\mathbb{Z}[x,y,z_+,z_-]$ generated by the following elements:
$$\begin{array}{c}
x^n-1,\  f_1(x,y)f_2(x,y),\  z_+z_--1-f_1(x,y)(2y+4f_3(x,y)),\\
 f_1(x,y)(z_+-1-2f_4(x,y)),\ f_1(x,y)(z_+-z_-).\\
 \end{array}$$

\begin{proposition}\label{3.9}
$R$ is isomorphic to the quotient ring $\mathbb{Z}[x,y,z_+,z_-]/I$.
\end{proposition}

\begin{proof}
By the definition of $R$, there exists a ring epimorphism $\phi$ from $\mathbb{Z}[x,y,z_+,z_-]$
onto $R$ such that
$$\phi(x)=[V(1,1)], \phi(y)=[V(2,0)], \phi(z_+)=[\O V(1,0)],\phi(z_-)=[\O^{-1}V(1,0)].$$
By \cite[Corollary 5.4(1), Lemma 5.6(1), Proposition 5.7]{ChHasLinSun} and Lemma \ref{3.4}, $\phi(I)=0$.
Hence $\phi$ induces a ring epimorphism $\ol{\phi}: \mathbb{Z}[x,y,z_+,z_-]/I\ra R$ such that
$\phi=\ol{\phi}\circ\pi$, where $\pi: \mathbb{Z}[x,y,z_+,z_-]\ra \mathbb{Z}[x,y,z_+,z_-]/I$
is the canonical epimorphism. Let $\ol{u}=\pi(u)$ for any $u\in\mathbb{Z}[x,y,z_+,z_-]$.
Then $\ol{x}^n=1$, $f_1(\ol{x}, \ol{y})f_2(\ol{x}, \ol{y})=0$,
$\ol{z_+}\ \ol{z_-}=1+f_1(\ol{x}, \ol{y})(2\ol{y}+4f_3(\ol{x}, \ol{y}))$, and
$f_1(\ol{x}, \ol{y})\ol{z_{\pm}}=f_1(\ol{x}, \ol{y})(1+2f_4(\ol{x}, \ol{y}))$.
The last equation above is equivalent to
$$\begin{array}{c}
\ol{y}^{n-1}\ol{z_{\pm}}=f_1(\ol{x}, \ol{y})(1+2f_4(\ol{x}, \ol{y}))
-\sum_{i=1}^{c(n-2)}(-1)^i\binom{n-1-i}{i}\ol{x}^i\ol{y}^{n-1-2i}\ol{z_{\pm}}.\\
\end{array}$$
Thus, by the above equations and the proof of \cite[Corollary 5.8]{ChHasLinSun}, one can see that
$\mathbb{Z}[x,y,z_+,z_-]/I$ is generated, as a $\mathbb Z$-module, by
$$\{\ol{x}^i\ol{y}^j, \ol{x}^i\ol{y}^l\ol{z_{\pm}}^m|0\<i\<n-1, 0\<j\<2n-2, 0\<l\<n-2, m\>1\}.$$
By Proposition \ref{3.5}, one can define a $\mathbb Z$-module map
$\psi: R\ra \mathbb{Z}[x,y,z_+,z_-]/I$ by
$\psi([V(1,1)]^i[V(2,0)]^j)=\ol{x}^i\ol{y}^j$ and
$\psi([V(1,1)]^i[V(2,0)]^l[\O^{\pm 1}V(1,0)]^m)=\ol{x}^i\ol{y}^l\ol{z_{\pm}}^m$,
where $0\<i\<n-1$, $0\<j\<2n-2$, $0\<l\<n-2$ and $m\>1$.
Now we have
$(\psi\circ\ol{\phi})(\ol{x}^i\ol{y}^j)=(\psi\circ\ol{\phi}\circ\pi)(x^iy^j)=\psi(\phi(x^iy^j))
=\psi([V(1,1)]^i[V(2,0)]^j)=\ol{x}^i\ol{y}^j$, and similarly
$(\psi\circ\ol{\phi})(\ol{x}^i\ol{y}^l\ol{z_{\pm}}^m)=\ol{x}^i\ol{y}^l\ol{z_{\pm}}^m$.
It follows that $\psi\circ\ol{\phi}={\rm id}$. Hence $\ol{\phi}$ is a ring monomorphism, and so
it is a ring isomorphism.
\end{proof}

Let $X=\{x, y, z_+, z_-, w_{m,\eta}|m\>1, \eta\in\ol{k}\}$, and let $\mathbb{Z}[X]$
be the corresponding polynomial ring.
Then $\mathbb{Z}[x,y,z_+,z_-]$ is a subring of $\mathbb{Z}[X]$. Take a subset $U$ of $\mathbb{Z}[X]$:
$$U=\left\{\left.\begin{array}{l}
x^n-1,\  f_1(x,y)f_2(x,y),\\
z_+z_--1-f_1(x,y)(2y+4f_3(x,y)),\\
f_1(x,y)(z_+-1-2f_4(x,y)),\ f_1(x,y)(z_+-z_-),\\
f_1(x,y)(w_{m,\eta}-m-mf_4(x,y)),\\
(z_+-f_4(x,y))w_{m,\eta}-mxf_1(x,y)^2,\\
(z_--f_4(x,y))w_{m,\eta}-mf_1(x,y)(y+f_3(x,y)),\\
w_{m,\eta}w_{s,\a}-msxf_1(x,y)^2,\\
w_{m,\eta}(w_{t,\eta}-1-f_4(x,y))-(t-1)mxf_1(x,y)^2\\
 \end{array}
\right|\begin{array}{l}
m, s, t\>1\\
\mbox{with } m\<t,\\
\eta, \a\in\ol{k}\\
\mbox{with } \eta\neq\a\\
\end{array}\right\},$$
where $f_1(x,y)$, $f_2(x,y)$, $f_3(x,y)$, $f_4(x,y)\in\mathbb{Z}[x,y]\subset\mathbb{Z}[x,y,z_+,z_-]\subset\mathbb{Z}[X]$
are given as before.
Let $I_0=(U)$, the ideal of $\mathbb{Z}[X]$ generated by $U$.

\begin{theorem}\label{3.10}
The Green ring $r(H_n(1,q))$ is isomorphic to $\mathbb{Z}[X]/I_0$.
\end{theorem}

\begin{proof}
We have already known that the ring $r(H_n(1,q))$ is commutative.
Hence by Proposition \ref{3.7},
there is a ring epimorphism $\psi$ from $\mathbb{Z}[X]$ onto $r(H_n(1,q))$ such that
$\psi(x)=[V(1,1)]$, $\psi(y)=[V(2,0)]$, $\psi(z_+)=[\O V(1,0)]$, $\psi(z_-)=[\O^{-1}V(1,0)]$
and $\psi(w_{m,\eta})=[M_m(1,0,\eta)]$ for any $m\>1$ and $\eta\in\ol{k}$.
By \cite[Corollary 5.4(1), Lemma 5.6(1), Proposition 5.7]{ChHasLinSun}, and Lemmas \ref{3.4} and \ref{3.6},
a straightforward verification shows that $\psi(u)=0$, $\forall u\in U$, and so $I_0\subseteq{\rm Ker}(\psi)$.
Hence there exists a unique ring epimorphism $\ol{\psi}$ from $\mathbb{Z}[X]/I_0$ onto $r(H_n(1,q))$
such that $\ol{\psi}(\ol{u})=\psi(u)$, $u\in\mathbb{Z}[X]$, where $\ol{u}$ is the image of
$u$ under the natural projection from $\mathbb{Z}[X]$ onto $\mathbb{Z}[X]/I_0$.
Obviously, $I\subseteq I_0\cap \mathbb{Z}[x,y,z_+,z_-]$ and $\psi(\mathbb{Z}[x,y,z_+,z_-])=R$,
where $I$ is given in Proposition \ref{3.9}. Therefore, one can define a ring map
$\s: \mathbb{Z}[x,y,z_+,z_-]/I\ra\mathbb{Z}[X]/I_0$ by $\s(u+I)=\ol{u}$,
$u\in\mathbb{Z}[x,y,z_+,z_-]$. Let us consider the composition:
$$f: \mathbb{Z}[x,y,z_+,z_-]/I\xrightarrow{\s}\mathbb{Z}[X]/I_0\xrightarrow{\ol{\psi}}r(H_n(1,q)).$$
Then ${\rm Im}(f)=R$. Hence $f$ can be viewed as a ring map from
$\mathbb{Z}[x,y,z_+,z_-]/I$ to $R$.
In this case, $f$ is exactly the ring isomorphism
$\ol{\phi}$ given in the proof of Proposition \ref{3.9}.
Hence $f$ is injective, so is $\s$. This implies that $I=I_0\cap\mathbb{Z}[x,y,z_+,z_-]$.
Moreover, the restriction $\ol{\psi}|_{\rm{Im}(\s)}$ can be regarded as a ring isomorphism
from $\rm{Im}(\s)$ to $R$.
Let $\widetilde{\psi}: R\ra\rm{Im}(\s)$ denote its inverse.

Now let $R_1$ be the $\mathbb Z$-submodule of $r(H_n(1,q))$ generated by
$$\{[V(1,1)]^i[V(2,0)]^l[M_m(1,0,\eta)]|0\<i\<n-1, 0\<l\<n-2, m\>1, \eta\in\ol{k}\}.$$
Then by Propositions \ref{3.5} and \ref{3.8}, it follows that the above set is a $\mathbb Z$-basis of $R_1$
and $r(H_n(1,q))=R\oplus R_1$. Hence there exists a $\mathbb Z$-module
map $\theta: r(H_n(1,q))\ra\mathbb{Z}[X]/I_0$ defined by $\theta(r)=\widetilde{\psi}(r)$ for $r\in R$,
and $\theta([V(1,1)]^i[V(2,0)]^l[M_m(1,0,\eta)])=\ol{x}^i\ol{y}^l\ol{w_{m,\eta}}$
for $0\<i\<n-1$, $0\<l\<n-2, m\>1$ and $\eta\in\ol{k}$. By the definition of $\s$, one can see that
${\rm Im}(\s)$ is generated, as a subring of $\mathbb{Z}[X]/I_0$, by $\ol{x}$, $\ol{y}$, $\ol{z_+}$ and $\ol{z_-}$.
By the definition of $I_0$, a straightforward computation similar to the proof of Proposition \ref{3.8}
shows that $\mathbb{Z}[X]/I_0$ is generated, as a $\mathbb Z$-module, by
$${\rm Im}(\s)\cup\{\ol{x}^i\ol{y}^l\ol{w_{m,\eta}}|0\<i\<n-1, 0\<l\<n-2, m\>1, \eta\in\ol{k}\}.$$
Obviously, $(\theta\ol{\psi})|_{\rm{Im}(\s)}={\rm id}_{\rm{Im}(\s)}$. For any
$0\<i\<n-1$, $0\<l\<n-2$, $m\>1$ and $\eta\in\ol{k}$, we have
$$\begin{array}{rl}
(\theta\ol{\psi})(\ol{x}^i\ol{y}^l\ol{w_{m,\eta}})=&\theta(\psi(x^iy^lw_{m,\eta}))\\
=&\theta([V(1,1)]^i[V(2,0)]^l[M_m(1,0,\eta)])=\ol{x}^i\ol{y}^l\ol{w_{m,\eta}}.\\
\end{array}$$
Thus, $\theta\ol{\psi}$ is the identity map on $\mathbb{Z}[X]/I_0$. Hence
$\ol{\psi}$ is injective. Consequently, $\ol{\psi}$ is a ring isomorphism.
\end{proof}

The stable Green ring was introduced in the study of Green rings for the modular
representation theory of finite groups. It is a quotient of the Green ring modulo
all projective representations (see \cite{Benson}). In our situation, it is easy to see
from \cite[Corollary 5.4(1)-(2), Lemma 5.6(2)]{ChHasLinSun} that the stable Green ring $St(H_n(1,q))$
of $H_n(1,q)$ is isomorphic to $r(H_n(1,q))/([V(n,0)])$, the quotient ring of the Green ring
$r(H_n(1,q))$ modulo the ideal $([V(n,0)])$ generated by $[V(n,0)]$.
Then by the proof of Theorem \ref{3.10}, one gets the following Corollary.

\begin{corollary}\label{3.11}
Let $\mathbb{Z}[X]$ be the polynomial ring in the following variables:
$$X=\{x, y, z, z^{-1}, w_{m,\eta}|m\>1, \eta\in\ol{k}\},$$
and let $J_s$ be the ideal of $\mathbb{Z}[X]$ generated by the following subset:
$$\left\{\left.\begin{array}{l}
x^n-1,\  f_1(x,y),\ zz^{-1}-1,\\
(z-f_4(x,y))w_{m,\eta},\
(z-z^{-1})w_{m,\eta},\\
w_{m,\eta}w_{s,\a}, \
w_{m,\eta}(w_{t,\eta}-1-f_4(x,y))\\
 \end{array}
\right|\begin{array}{l}
m, s, t\>1 \mbox{ with } m\<t,\\
\eta, \a\in\ol{k} \mbox{ with } \eta\neq\a\\
\end{array}\right\},$$
where $f_1(x,y), f_4(x,y)\in\mathbb{Z}[x,y]\subset\mathbb{Z}[X]$ are given as before.
Then the stable Green ring $St(H_n(1,q))$ of $H_n(1,q)$ is isomorphic to the
quotient ring $\mathbb{Z}[X]/J_0$.
\end{corollary}
\vspace{1cm}

\centerline{ACKNOWLEDGMENTS}

This work is supported by NNSF of China (No. 11571298).\\

\end{document}